\documentclass[11pt]{article}
\usepackage{amssymb,amsmath,setspace,graphicx,multicol,wrapfig,amsfonts,amsthm,titling,url,array,parskip,enumerate}
\usepackage{hyperref}
\usepackage{placeins}
\usepackage{aurical,pbsi}
\usepackage[T1]{fontenc}
\usepackage[hmargin=2.5cm,vmargin=2.5cm]{geometry}
\makeatletter
\def\thm@space@setup{%
  \thm@preskip=0.1in
  \thm@postskip=0in%\thm@preskip % or whatever, if you don't want them to be equal
}
\makeatother

\numberwithin{equation}{section}

\theoremstyle{plain}
\newtheorem{thm}{Theorem}[section]
\newtheorem{lemma}[thm]{Lemma}

\newtheorem{cor}[thm]{Corollary}

\theoremstyle{definition}
\newtheorem{defn}{Definition}[section]

\newtheorem{rem}{Remark}[section]

\theoremstyle{remark}

\newtheorem*{example}{Example}

\DeclareMathOperator{\wt}{wt}
\DeclareMathOperator{\pwt}{p\,wt}
\DeclareMathOperator{\frow}{f\,row}
\DeclareMathOperator{\fcol}{f\,col}
\DeclareMathOperator{\sgn}{sgn}

\begin{document}
%\maketitle
\begin{center} {\Large{\sc LGV Proof of a Determinantal Theorem for TASEP Probabilities}} \\
\vspace{0.1in}
Olya Mandelshtam
 \end{center}

\begin{abstract}
The Totally Asymmetric Simple Exclusion Process (TASEP) is a non-equilibrium particle model on a finite one-dimensional lattice with open boundaries. In our earlier paper, we obtained a determinantal formula that computes the steady state probabilities of this process by the enumeration of ``Catalan alternative tableaux'', which are certain fillings of Young diagrams. Here, we present a new, more illuminating bijective proof of this determinantal formula using the Lindstr\"{o}m-Gessel-Viennot Lemma.
\end{abstract}

\section{The TASEP and Catalan alternative tableaux}

%Define TASEP and \tau word

The totally asymmetric exclusion process (TASEP) is a version of the well-studied non-equilibrium particle model from statistical mechanics. In a TASEP of size $n$, particles hop to the right along a line with locations labelled 1 through $n$ such that there is at most one particle per location, and particles can enter at location 1 and exit from location $n$ with respective rates $\alpha$ and $\beta$. A state of the TASEP is described by a word $\tau = \tau_1\ldots\tau_n$ in $\{0,1\}^n$ where $\tau_i=1$ if there is a particle in location $i$, and $\tau_i=0$ otherwise. The discrete Markov process can be defined as follows, for $\tau'$ and $\tau''$ arbitrary words in $\{0,1\}^{\ast}$: at each time step, 
%put harpoons with rates there
\[
0\tau'  \overset{\alpha}{\rightharpoonup} 1\tau' \qquad \tau'1  \overset{\beta}{\rightharpoonup} \tau'0 \qquad \tau'10\tau''  \overset{1}{\rightharpoonup} \tau'01\tau''
\]
where $X  \overset{p}{\rightharpoonup} Y$ means that the transition from $X$ to $Y$ occurs with probability $\frac{p}{n+1}$ where $n$ is the length of $X$ and $Y$. We use the notation $Pr(\tau)$ to denote the stationary probability of state $\tau$.

\begin{wrapfigure}[9]{r}{0.33\textwidth}
\centering
\includegraphics[width=0.3\textwidth]{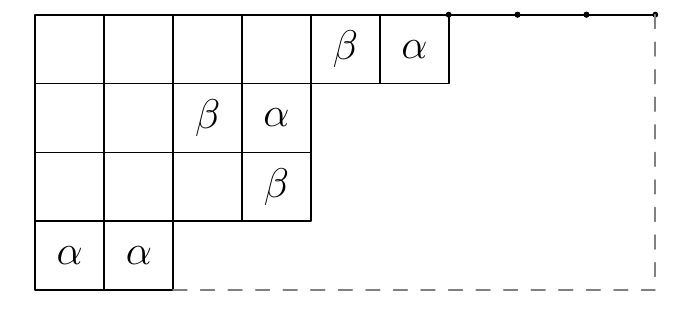}
\caption{A Catalan tableau of size $(13,4)$, type 0001001100100, and weight $\alpha^8\beta^{12}$.}
\noindent
%\hrulefill
\label{cat_example}
\end{wrapfigure}

A large number of related combinatorial objects interpret the stationary probabilities of the TASEP, and an explicit determinantal formula was given in \cite{mandelshtam} for $\Pr(\tau)$. In this work, we present a simple bijective proof for that formula that uses the Lindstr\"{o}m-Gessel-Viennot Lemma. We begin by defining the main object, the Catalan alternative tableau. 

%Define Catalan tableau as in fixed paper
\begin{defn}\label{catalan}
A \textbf{Catalan alternative tableau} $T$ of \textbf{size} $(n,k)$ is a Young diagram $Y=Y(T)$ contained in a $k \times n-k$ rectangle, justified to the northwest, together with a filling of the boxes of $Y$ with $\alpha$'s and $\beta$'s according to the following rules:
\begin{enumerate}[i]
\item Every box in the same column and above an $\alpha$ must be empty.
\item Every box in the same row and left of a $\beta$ must be empty.
\item Every box that does not lie above an $\alpha$ or left of a $\beta$ must contain an $\alpha$ or a $\beta$.
\end{enumerate}
\end{defn}

We associate to $T$ a lattice path $L=L(T)$ with steps south and west, which starts at the northeast corner of the rectangle and ends at the southwest corner, and follows the southeast border of $Y$. The \textbf{type} of $T$ is the word $\tau$ in $\{0,1\}^{\ast}$ that we obtain by reading $L$ from northeast to southwest and assigning a 1 to a south-step and a 0 to a west-step.

We also associate to $Y$ the partition $\lambda=\lambda(T)=(\lambda_1,\ldots,\lambda_k)$, where $\lambda_1 \geq \lambda_2 \geq \cdots \geq \lambda_k \geq 0$ and $\lambda_i$ is the number of boxes in $Y$ in the $i$th row of the rectangle. We call $\lambda$ the \textbf{shape} of $T$ and of $Y$. 

\begin{defn}
For a Catalan alternative tableau $T$ with associated shape $\lambda(T)$ and of type $\tau$ with exactly $k$ 1's, we make the equivalent definition $\lambda(\tau)=(\lambda_1,\ldots,\lambda_k)$.  Here $\lambda_i$ is the number of 0's following the $i$'th 1 in the word $\tau$. It is easy to see that $\lambda(T) = \lambda(\tau)$.
\end{defn}

\begin{defn}
For a Catalan alternative tableau $T$ of size $(n,k)$ and its associated Young diagram $Y$, an \textbf{$\alpha$-free} column is a column of the $k \times n-k$ rectangle containing $Y$ that does not contain an $\alpha$. Similarly, a \textbf{$\beta$-free} row is a row of the rectangle that does not contain a $\beta$.
\end{defn}

\begin{defn}\label{weight1}
The \textbf{weight} of a Catalan alternative tableau $T$ of size $(n,k)$ with associated Young diagram $Y$ is
\[
\wt(T) = (\alpha\beta)^n\left(\frac{1}{\alpha}\right)^{\fcol(T)}\left(\frac{1}{\beta}\right)^{\frow(T)}
\]
where $\fcol(T)$ is the number of $\alpha$-free columns and $\frow(T)$ is the number of $\beta$-free rows in $T$.
\end{defn}

\begin{example}
Figure \ref{cat_example} shows a Catalan alternative tableau of type 0001001100100, weight $\alpha^8\beta^{12}$ with $\lambda(T) = (6,4,4,2)$.
\end{example}

Catalan alternative tableaux are a specialization of ``alternative tableaux'', which are more general objects connected to the combinatorics of the PASEP (an exclusion process in which particles are allowed to hop both left and right). The alternative tableaux are related by simple bijections to several similar objects in the literature, such as permutation tableaux, staircase tableaux, and tree-like tableaux. Some references for these objects can be found in \cite{nadeau, treelike} (thus we could similarly define Catalan permutation tableaux, Catalan staircase tableaux, and Catalan tree-like tableaux). Furthermore, connections to the combinatorics of the (P)ASEP are explored in an expansive collection of slides from lectures given by Viennot in \cite{slides}. The author recommends these slides and references therein as a source for further reading on this topic. For the rest of this paper, we will refer to the Catalan alternative tableaux simply as ``Catalan tableaux''.

% Determinantal Theorem

\begin{thm}\label{deter}
Let $\tau$ be a word in $\{0,1\}^n$ with $k$ 1's representing a state of the TASEP of length $n$ with exactly $k$ particles. Let $\lambda(\tau) = (\lambda_1,\ldots,\lambda_k)$ be the partition associated to $\tau$. Define $P(\tau) = \sum_T wt(T)$ where the sum is over Catalan tableaux $T$ of type $\tau$. Then
\[
P(\tau) = \alpha^{k+\lambda_1}\beta^n \det A^{\alpha,\beta}_{\lambda(\tau)}
\]
where $A^{\alpha,\beta}_{\lambda(\tau)} = (A_{ij})_{1 \leq i,j \leq k}$ with
\[
A_{ij} = \left({\lambda_{j+1} \choose j-i+1} + \frac{1}{\beta}{\lambda_{j+1} \choose j-i} \right) + \sum_{p=1}^{\lambda_j-\lambda_{j+1}} \left(\frac{1}{\alpha}\right)^p \left( {\lambda_{j+1}+p-1 \choose j-i} + \frac{1}{\beta}{\lambda_{j+1} + p-1 \choose j-i-1} \right).
\]
\end{thm}

%CW theorem

We give the proof in Section \ref{proof}. Instead, here we demonstrate the importance of the above in connection with the TASEP. From the following theorem, we obtain that $P(\tau)$ plays a central role in computing the stationary probability of state $\tau$ of the TASEP.

\begin{thm}[Corteel, Williams (2007)\footnote{Duchi and Schaeffer in \cite{duchi}  provide the first combinatorial interpretation for probabilities of the TASEP, which is in terms of a different combinatorial object that resembles two rows of particles. However, an important advantage of this theorem is that its formulation in terms of tableaux can be extended to the setting of the ASEP, which is a more general exclusion process in which particles can hop both left and right, as well as hop in and out of the lattice from both sides. See \cite{williams2007,williams2011} for details.}] \label{cw}
Let $Z_n = \sum_{T'} wt(T')$ be the sum of the weights of all Catalan tableaux $T'$ of size $n$. Then
\begin{displaymath}
Pr(\tau) = \frac{1}{Z_n}\sum_T wt(T)
\end{displaymath}
where the sum is over Catalan tableaux $T$ of type $\tau$.
\end{thm}

Furthermore, from Derrida et.\ al.\ \cite{derrida}, we have the following formula.
\begin{equation}\label{Z}
Z_n = (\alpha\beta)^n \sum_{p=1}^n \frac{p}{2n-p} {2n-p \choose n} \frac{\alpha^{-p-1}-\beta^{-p-1}}{\alpha^{-1}-\beta^{-1}}.
\end{equation}

%Corollary: probability of TASEP

Thus we obtain the final result.

\begin{cor}
The stationary probability of state $\tau$ of a TASEP of size $n$ is
\[
Pr(\tau) = \frac{\alpha^{k+\lambda_1}\beta^n}{Z_n} \det A^{\alpha,\beta}_{\lambda}.
\]
where $\lambda=\lambda(\tau)$ and $Z_n$ is of Equation \ref{Z}.
\end{cor}

\begin{proof}
The Corollary follows immediately from Theorem \ref{deter} and Theorem \ref{cw}.
\end{proof}

\section{Bijection from Catalan tableaux to weighted Catalan paths}

In this section, we present a canonical bijection from a filling of the Catalan tableau with associated Young diagram $Y$ to a lattice path on a Young diagram of the same shape. Viennot describes an analogous bijection from Catalan permutation tableaux (which are in bijection to the Catalan tableaux) to weighted lattice paths in \cite{viennot}. We reformulate this bijection for the Catalan tableaux and assign the weights to the resulting lattice path in a particular way.

%put new weights of 1/a, 1/b on the edges and prove that weight is the same

%We now give a definition for the weight of a Catalan tableau that is equivalent to Definition \ref{weight1}.
%\begin{defn}\label{weight2}
%Let $T$ be a Catalan tableau of size $(n,k)$ with associated Young diagram  $Y$. Then
%\[
%wt2(T) = (\alpha\beta)^n\left(\frac{1}{\alpha}\right)^s\left(\frac{1}{\beta}\right)^t
%\]
%where $s$ is the number of $\alpha$-free columns and $t$ is the number of $\beta$-free rows in the filling of $Y$. 
%\end{defn}

%\begin{lem}
%For a Catalan tableau $T$ of size $(n,k)$, with $wt(T)$ the weight from Definition \ref{weight},
%\[
%wt(T) = wt2(T). 
%\]
%\end{lem}

%\begin{proof}
%Let $Y(T)$ be the Young diagram associated to $T$. Let $j$ be the number of $\alpha$'s and $\ell$ the number of $\beta$'s in the filling of $Y$. Since each row of the $k \times n-k$ rectangle can have at most one $\beta$, the number of $\beta$-free rows of the rectangle is $k-\ell$. Similarly, since each column of the rectangle can have at most one $\alpha$, the number of $\alpha$-free columns of the rectangle is $n-k-j$. Thus 
%\[
%wt2(T) = (\alpha\beta)^n\left(\frac{1}{\alpha}\right)^{n-k-j}\left(\frac{1}{\beta}\right)^{k-\ell} = \alpha^{k+j} \beta^{n-k+\ell} = wt(T).
%\]
%\end{proof}

%From here on, we use $wt2(T)$ to describe the weight of a Catalan tableau.

\subsection{Weighted Catalan path}
Let $Y$ be a Young diagram contained within a $k \times n-k$ rectangle. A lattice path \emph{constrained by} $Y$ is a path that begins in the northeast corner and ends at the southwest corner of rectangle, and takes the steps south and west in such a way that it never crosses the southeast boundary of $Y$. 

\begin{figure}[h]
\centering
\includegraphics[width=\textwidth]{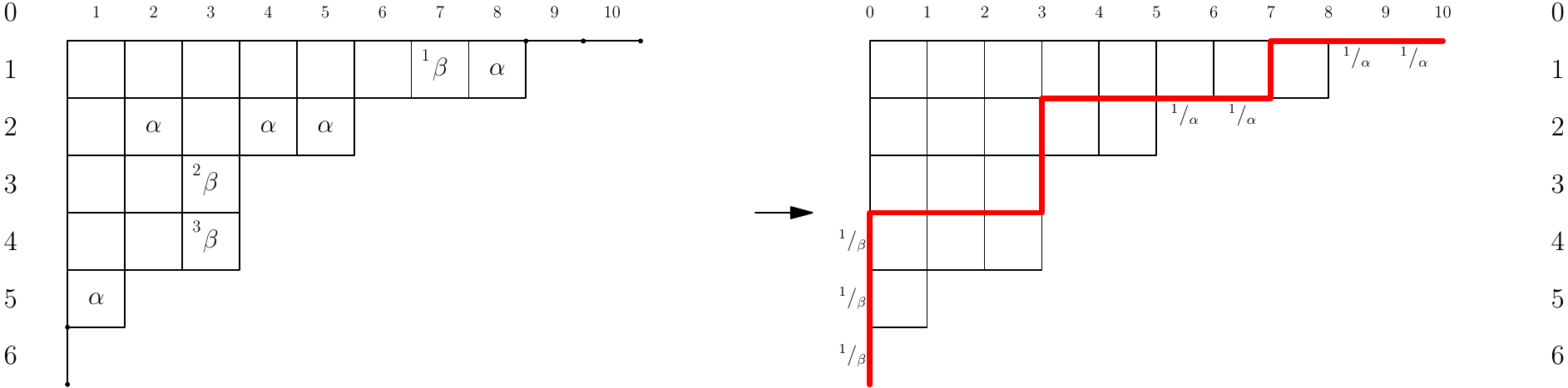}
\caption{A Catalan tableau $T$ and its corresponding weighted Catalan path $C$ on a tableau of shape $\lambda=(8,5,3,3,1,0)$ and weight $\wt(T)=\alpha^{12}\beta^{13}$. On the left, the $\beta$'s are labeled such as to generate the partition $(7,3,3,0,0,0)$ where the column containing the $i$th beta is the length of the $i$th row of the partition. This partition is precisely the shape of the path in the figure on the right. The path weight of $C$ is $\pwt(C)=\left(\frac{1}{\beta}\right)^3\left(\frac{1}{\alpha}\right)^4$, and so $\wt(C) = (\alpha\beta)^{16}\left(\frac{1}{\beta}\right)^3\left(\frac{1}{\alpha}\right)^4 = \wt(T)$.}
\noindent
%\hrulefill
\label{cat_path}
\end{figure}

\begin{defn}
A \emph{Catalan path} $C$ of size $(n,k)$ with associated Young diagram $Y$ is a lattice path constrained by $Y$ with the following weights on its edges:
\begin{itemize}
\item A south edge that coincides with the east border of the rectangle receives a $\frac{1}{\beta}$.
\item A south edge that does not coincide with the east border of the rectangle receives a 1.
\item A west edge that coincides with the south boundary of Y receives a $\frac{1}{\alpha}$.
\item A west edge that does not coincide with the south boundary of Y receives a 1.
\end{itemize}
\end{defn}

\begin{defn}
The \textbf{path weight} $\pwt(C)$ of the Catalan path $C$ is the product of the weights on its edges. We call the total weight of the Catalan path $\wt(C)$, with $\wt(C)=(\alpha\beta)^n\pwt(C)$.
\end{defn}

The following lemma describes a natural correspondence between the Catalan tableaux and the Catalan paths.

\begin{lemma}
There is a weight-preserving bijection between the set of Catalan paths of size $(n,k)$ constrained by the Young diagram $Y$ to the set of Catalan tableaux of size $(n,k)$ of type $\tau$ such that $\lambda(\tau)$ is the same partition that describes $Y$.
\end{lemma}

\begin{proof}
Let a Catalan path $C$ of size $(n,k)$ constrained by a Young diagram $Y$ of shape $\lambda=(\lambda_1,\ldots,\lambda_k)$ be described by the partition $(C_1,\ldots,C_k)$ that is weakly smaller than $\lambda$. In other words, $C_1 \geq C_2 \cdots \geq C_k$ and $0 \leq C_i \leq \lambda_i$, and $C_i$ is the position of the south step of the lattice path that occurs in row $i$ of the $k \times n-k$ rectangle. 

We map $(C_1,\ldots,C_k)$ to a Catalan tableau $T$ as follows. First we label the columns of the $k \times n-k$ rectangle with 1 through $n-k$ from left to right. Then, for $i$ in $\{1,\ldots,k\}$, if $C_i>0$, we place a $\beta$ in column $C_i$ of $Y$ such that it is the south-most position possible with the condition that there is at most one $\beta$ per row. We now place an $\alpha$ in the lowest possible $\beta$-free row of every column. (Consequently, a column does not receive an $\alpha$ if and only if it has zero $\beta$-free rows.) It is easy to check that this construction results in a valid Catalan tableau. 

Conversely, to map a Catalan tableau $T$ to the partition $(C_1,\ldots,C_k)$, we label the $\beta$'s in the filling of $Y$ from left to right and top to bottom with $1,\ldots,\ell$ where $\ell$ is the number of $\beta$'s, and we let $C_i$ be the label of the column containing the $i$'th beta. We let $C_{\ell+1}=\cdots=C_k=0$. In this construction, the labels on the $\beta$'s decrease as the labels on the columns decrease, as in the left image of Figure \ref{cat_path}, so $C_i \geq C_{i+1}$. The partition $(C_1,\ldots,C_k)$ is then directly mapped to the Catalan path $P$.

Now we show the weight $\wt(C)$ of the Catalan path $C$ is the same as the weight $\wt(T)$ of the Catalan tableau $T$.
Let $\{C_{i_1},\ldots,C_{i_m}\}$ be the subset of $\{C_1,\ldots,C_k\}$ that represents the south steps that touch the south boundary of $Y$. Then the contribution of the $\left(\frac{1}{\alpha}\right)$ to the weight of the path is $\prod_{j=1}^m \left(\frac{1}{\alpha}\right)^{C_{i_j}-\lambda_{i_j+1}}$. This is because, for each $j$, if $C_{i_j}$ touches the south boundary of $Y$, we know that there are zero $\beta$-free rows in the column $i_j$. In particular, no column of the Catalan tableau between $\lambda_{i_j+1}$ and $i_j$ can contain an $\alpha$, so every west-edge of the path in those columns carries a weight of $\frac{1}{\alpha}$. It follows that both the Catalan tableau and the Catalan path have the same power of $\frac{1}{\alpha}$ contributed to their weight.

As for the factor of $\frac{1}{\beta}$, by the construction of the path, it must be $\left(\frac{1}{\beta}\right)^t$, where $t$ is the number of $C_j$ that equal 0. But we already know that if $C_j=0$, it means that row $j$ of the Catalan tableau is $\beta$-free, and so contributes a $\frac{1}{\beta}$ to the weight of the tableau. Thus $\wt(C)=\wt(T)=(\alpha\beta)^n\left(\frac{1}{\beta}\right)^t\prod_{j=1}^m \left(\frac{1}{\alpha}\right)^{C_{i_j}-\lambda_{i_j+1}}$ where $t,i_1,\ldots,i_m$ were defined in the above paragraphs.
\end{proof}

%\subsection{Weighted Type 2 Catalan path}
%We now make a simple modification to the assignment of weights to the edges of the Catalan paths by placing all the non-unity weights on the south edges of the path, while preserving the total weight. 

%%First, we extend the shape by adding row of length $n-k$ to the top of the $k \times n-k$ rectangle containing the tableau, and we extend the lattice path by adding to it a single vertical edge to its northeast corner. We call this the Type 2 Catalan path, and we can see an example of such in the right-most image of Figure \ref{cat_path}. Any Type 2 Catalan path will necessarily have to begin with an immediate south edge, and the portion of the path following this auxiliary edge is precisely the original Catalan path. In particular, the only purpose the auxiliary edge serves is to facilitate the assignment of weights. We call the new row of length $n-k$ row 0, and we add a column label $C_0=n-k$ for the auxiliary south edge of the path. The south boundary of the augmented shape is identical to the south boundary of $Y$.

\section{Bijection from a weighted lattice path on a Young diagram of $k$ rows to $k$ disjoint weighted paths}

Let $D$ be a digraph where we assume finitely many paths between any two vertices. Let $\textbf{e}=(e_1,\ldots,e_k)$ and $\textbf{v}=(v_1,\ldots,v_k)$ be $k$-tuples of vertices of $D$. Let every edge of $D$ be assigned a weight.

\begin{defn} A \textbf{$k$-path} from $\textbf{e}$ to $\textbf{v}$ is a $k$-tuple of paths $\textbf{P}(\textbf{e},\textbf{v})=(P_1,\ldots,P_k)$ where for some fixed $\pi \in S_k$, $P_i$ is a path from $e_i$ to $v_{\pi(i)}$. The $k$-path $\textbf{P}$ is \textbf{disjoint} if the paths $P_i$ are all vertex disjoint. 
\end{defn}

\begin{defn}
The \textbf{weight} $\wt(P_i)$ of a path $P_i$ is the product of the weights on its edges. The weight $\wt(\textbf{P})$ of the $k$-path $\textbf{P}=(P_1,\ldots,P_k)$ is the sum of the weights of its components, in other words $\wt(\textbf{P}) = \sum_{i=1}^k \wt(P_i)$.
\end{defn}

\begin{thm}[Lingstr\"{o}m, Gessel-Viennot]\label{LGV}
Let $D$ be a digraph, and let $\textbf{u}=(u_1,\ldots,u_k)$ and $\textbf{y}=(y_1,\ldots,y_k)$ be $k$-tuples of vertices of $D$. Let $\mathcal{P}_{ij}$ be the set of paths from $u_i$ to $y_j$. Define $w_{ij}=\sum_{p \in \mathcal{P}_{ij}} \wt(p)$. Then
\[
\sum_{\pi \in S_k} \sum_\textbf{P} \sgn(\pi) \wt(\textbf{P}) = \det \left( w_{ij} \right)_{1 \leq i,j \leq k}.
\]
where $\textbf{P}$ ranges over all disjoint $k$-paths $\textbf{P}(\textbf{u},\pi(\textbf{y}))$.
\end{thm}

In this section, we provide a bijection from a Catalan path on a Young diagram $Y$ to a disjoint $k$-path on a corresponding digraph with appropriately assigned weights on the edges. Ignoring the weights, we obtain the canonical bijection from lattice paths constrained by a Young diagram to disjoint $k$-paths.\footnote{We can treat the Catalan path and the Young diagram that contains it simply as nested lattice paths. The duality of nested lattice paths with disjoint $k$-paths is known in the literature and is described as the Kreweras-Narayana determinant. In particular, this duality is described in slides by Viennot \cite{slides2}, and the case for $\alpha=\beta=1$ of our problem is solved therein.} 
This bijection allows us to enumerate the Catalan paths as an application of the Lindstr\"{o}m-Gessel-Viennot Lemma.

%lattice path on YD is represented by a collection of south steps, exactly one in each row

Let $C$ be a Catalan path of size $(n,k)$ with associated Young diagram $Y$ of shape $\lambda=(\lambda_1,\ldots,\lambda_k)$. We label the vertical lines in the $k \times n-k$ rectangle from left to right with $\{0,1,\ldots\,n-k\}$. Let $C$ be described by the partition $(C_1,\ldots,C_k)$ where $C_i$ is the label of the south-step of $C$ in row $i$. Since $C$ consists of only south- and west- steps, we necessarily have $C_1 \geq \cdots \geq C_k \geq 0$. 

%define twisted tableau from YD

Now we define a twisted tableau $\tilde{Y}$ from $Y$ as follows: for $1 \leq i \leq k$, draw a row of $\lambda_i$ parallelograms consisting of east and southeast edges, and left-justify the rows as in the middle image of Figure \ref{LGV_bijection}. In each row, we label the southeast edges of the parallelograms with $0,1,2,\ldots$ from left to right. We put weights on the edges of the parallelograms in the following way: 
\begin{itemize}
\item the edges with label 0 receive a $\frac{1}{\beta}$,
\item otherwise if an edge in row $i$ has label $t$ and $t>\lambda_{i+1}$, the edge receives a $\left(\frac{1}{\alpha}\right)^{t-\lambda_{i+1}}$.
\end{itemize}
Every other edge receives a weight of 1. 

We mark the left-most vertices of each row of parallelograms as the $k$ special points $e_1,\ldots,e_k$ from top to bottom. We also mark the right-most vertices of each row of parallelograms as the $k$ special points $v_1,\ldots,v_k$. Finally, we convert $\tilde{Y}$ into a digraph by directing all its edges from northwest to southeast. We denote by $\mathcal{P}_{ij}$ the set of weighted paths from $e_i$ to $v_j$.

%Build paths e_i to v_i on the twisted tableau: a path with south step at p^th edge is equivalent to the choice of a down-step in column p in row i of YD. 

\begin{figure}[h]
\centering
\includegraphics[width=\textwidth]{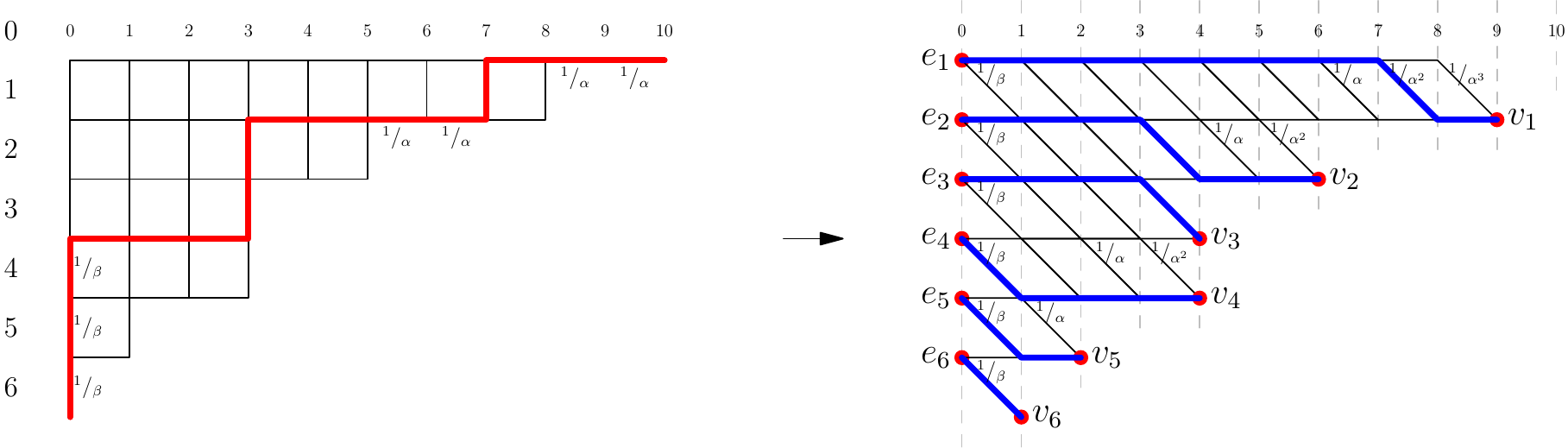}
\caption{A Catalan path represented by partition $(7,3,3,0,0,0)$ on a Young Diagram with rows $(C_1,\ldots,C_6)=(1,\ldots,6)$ and the corresponding set of paths $\{p_{ii}\}_{1 \leq i \leq 6}$ where $p_{ii} \in \mathcal{P}_{ii}$ has a single diagonal step at edge labeled $C_i$. This Catalan path is the same one as in Figure \ref{cat_path}.}
\noindent
%\hrulefill
\label{LGV_bijection}
\end{figure}

We map the partition $(C_1,\ldots,C_k)$ on $Y$ to a $k$-path $\textbf{P}(C) = \textbf{P}(\textbf{e},\textbf{v})$ on $\tilde{Y}$ in the following way. We write $\textbf{P}(C) = (p_{11},\ldots,p_{kk})$ where $p_{ii} \in \mathcal{P}_{ii}$. For each $i$ in $\{1,\ldots,k\}$, we define $p_{ii}$ as follows: let the single diagonal step in $p_{ii}$ be the southeast edge in row $i$ with label $C_i$. The rest of the edges in $p_{ii}$ must necessarily be the horizontal edges that connect that diagonal step from $e_i$ to $v_i$. From Figure \ref{LGV_bijection}, it is easy to see this is a one to one correspondence. 

\begin{rem} \label{segment}
It is important to note that the segment of $C$ that lies in the columns $\{\lambda_1+1,\ldots,n-k\}$ is ignored in the construction of $\textbf{P}(C)$. This is permissible since any Catalan path constrained by $\lambda$ must necessarily have the same such segment. Thus it suffices to simply adjust the weight of $\textbf{P}(C)$ by the weight contribution of that segment, which is $\left(\frac{1}{\alpha}\right)^{n-k-\lambda_1}$. %We can find an example of this in Figure \ref{LGV_bijection}. 
\end{rem}

\begin{lemma} Based on the construction of the $k$-path $\textbf{P}(C)$ above, we claim that (i.) $\textbf{P}(C)$ is disjoint if and only if $C_1 \geq \cdots \geq C_k$ and (ii.) $\pwt(C) = \left(\frac{1}{\alpha}\right)^{n-k-\lambda_1} \wt(\textbf{P}(C))$.

%the following:
%\begin{enumerate}[i.]
%\item $\textbf{P}(C)$ is disjoint if and only if $C_1 \geq \cdots \geq C_k$.
%\item $\pwt(C) = \wt(\textbf{P}(C))$.
%\end{enumerate}
\end{lemma}

\begin{proof}\ [\textit{i.}] It is easy to see from the construction that $C_i \geq C_{i+1}$ if and only if the diagonal edge in row $i$ is strictly to the right of the diagonal edge in row $i+1$. That implies $p_{ii}$ is strictly to the northeast of $p_{i+1\ i+1}$. Since the $p_{ii}$'s are nested paths, this implies $\textbf{P}(C)$ is disjoint.

[\textit{ii.}] We prove the equality by comparing  $\wt(p_{ii})$ to the weight contribution of the segment of $C$ that is in row $i$ (including the south border of the row), and showing they are equal for each $1 \leq i \leq k$. 
\begin{itemize}
\item First, if $C_i=0$, then $\wt(p_{ii})=\frac{1}{\beta}$, and also the weight contribution of row $i$ in $C$ is $\frac{1}{\beta}$. See rows 3-6 in the example in Figure \ref{LGV_bijection}.

\item When $C_i>0$, there is no contribution of $\frac{1}{\beta}$ to the segment of $C$ in row $i$ or to $p_{ii}$, so we consider only the contribution of $\frac{1}{\alpha}$. If $0 < C_i \leq \lambda_{i+1}$, the south-step of $C$ in row $i$ does not touch the south boundary of $Y$, so there is no contribution of $\frac{1}{\alpha}$ from that segment of the path, and hence the total weight contribution is 1. Similarly, $p_{ii}$ does not contain any edges with non-unit weight and so $\wt(p_{ii})=1$. See rows 2-3 in the example in Figure \ref{LGV_bijection}.

\item If $C_i>\lambda_{i+1}$, the south-step of $C$ in row $i$ touches the south boundary of $Y$, so that segment of the path has $C_i-\lambda_{i+1}$ west-edges that coincide with the south boundary of $Y$ and thus carry the weight $\frac{1}{\alpha}$. Thus the total contribution to the weight of the segment of $C$ in row $i$ is $\left(\frac{1}{\alpha}\right)^{C_i - \lambda_{i+1}}$. By the construction, $p_{ii}$ has weight $\left(\frac{1}{\alpha}\right)^{C_i - \lambda_{i+1}}$ on its diagonal edge, and that also equals $\wt(p_{ii})$. See row 1 in the example in Figure \ref{LGV_bijection}.
\end{itemize}
From the above, for each $i$, the contribution of the weight of the segment of $C$ in row $i$ equals $\wt(p_{ii})$. By Remark \ref{segment}, we have excluded from $\textbf{P}(C)$ the contribution of the weight of the segment of $C$ that lies to the northeast of $Y$. Consequently, we have $\pwt(C)=\left(\frac{1}{\alpha}\right)^{n-k-\lambda_i}\wt(\textbf{P}(C))$ as desired.
\end{proof}

%k paths {e_i} to {v_i} nonintersecting iff they go e_i to v_i and columns of south-steps are p_1\geq p_2 \geq p_k, which is equivalent to that set of south steps on YD that form a lattice path

%\begin{lem}
%For a left-justified stack of parallelograms as in Figure \ref{BLAH} with left-most vertices labeled $\{e_i\}$ and right-most vertices labeled $\{v_i\}$, and southeast edges labeled $0,1,2,\ldots$ from left to right for each row., 
%For the above construction, the $k$ paths from $e_i$ to $v_i$ for $i=1,\ldots,k$ are nonintersecting if and only if $C_1 \geq \cdots \geq C_k$.
%\end{lem}

%Since we have put weights only on the vertical edges of the Catalan path, and these edges correspond directly to the southeast edges on the diagram of parallelograms which have the same weights, it is immediate that the product of the weights of the non-intersecting paths $\{p_{ii}\}_{0 \leq i \leq k}$ is the same as the weight of the Catalan path.

\subsection{Proof of Theorem \ref{deter}}\label{proof}

We make the simple observation that a $k$-path $(P_i,\ldots,P_k)$ from the $\textbf{e}$ to $\textbf{v}$ is disjoint if and only if each path $P_i$ is from $e_i$ to $v_i$. As before, let $w_{ij}=\sum_{p \in \mathcal{P}_{ij}} \wt(p)$ for $\mathcal{P}_{ij}$ the collection of paths from $e_i$ to $v_j$. Then from the bijection above and from Theorem \ref{LGV}, we obtain
\[
\sum_C \pwt(C) = \left(\frac{1}{\alpha}\right)^{n-k-\lambda_i} \sum_{\textbf{P}} \wt(\textbf{P}) = \left(\frac{1}{\alpha}\right)^{n-k-\lambda_i} \det \left( w_{ij} \right)_{1 \leq i,j \leq k},
\]
where $C$ ranges over the Catalan tableaux constrained by $Y$, and $\textbf{P}$ ranges over the disjoint $k$-paths from $\textbf{e}$ to $\textbf{v}$ on $\tilde{Y}$.

%LGV determinant gives weight

%Therefore, the weight generating function for the Catalan paths that are constrained by the shape $\lambda$ in the rectangle $k \times n-k$ can be computed by the LGV determinant that computes the weight generating function for the product of the weights of non-intersecting paths from $\{e_i\}$ to $\{v_i\}$ for $0 \leq i \leq k$. In other words, if $wt(p_{ij})$ is the weight of a path $p_{ij}$ from $e_i$ to $v_j$, then
%\[
%\sum_P wt(P) = \prod_{i=0}^k wt(p_{ii}) = \det P_{\lambda}
%\]
%where the first sum is over paths $P$ constrained by the shape $\lambda$, and 
%$P_{\lambda} = p_{00}\begin{pmatrix} p_{11}&p_{12}&\cdots&p_{1k}\\
%p_{21}&p_{22}&\cdots&p_{2k}\\
%\vdots&\vdots&\vdots&\vdots\\
%p_{k1}&p_{k2}&\cdots&p_{kk}.
%\end{pmatrix}
%$.

%p_e_i to v_j calculation
It is not difficult to check that $w_{ij}$ for $i,j>0$ equals precisely the entry $A_{ij}$ from Theorem \ref{deter}. We describe the calculations below.

Consider the paths from $e_i$ to $v_j$ that have weight generating function $w_{ij}$. First, if $i>j+1$, there are zero such paths since all paths can only take east and southeast steps. Next, if $i=j+1$, there is exactly one path, namely the one that takes only horizontal steps from $e_i$, and so the weight on that path is 1, and thus $w_{i,i-1}=1$. Finally, assume $i\leq j$. Then any path in $\mathcal{P}_{ij}$ takes $j-i+1$ southeast steps, of which at most one step could have a weight of $\frac{1}{\beta}$, and at most one other step could have a weight of $\left(\frac{1}{\alpha}\right)^{\ell}$ for some $\ell>0$. Thus we count four cases for paths in $\mathcal{P}_{ij}$: 
\begin{enumerate}
\item \emph{A path has all its steps of weight 1.} The path necessarily takes the first step east and goes to the right-most vertex of parallelogram number $\lambda_{i+1}$ in the $i$th row. This can happen in ${\lambda_{i+1} \choose j-i+1}$ ways, and every such path has weight 1.

\item \emph{A path has one step of weight $\frac{1}{\beta}$ and the rest of weight 1.} The path necessarily takes the first step southeast and goes to the right-most vertex of parallelogram number $\lambda_{i+1}$ in the $i$th row. This can happen in ${\lambda_{i+1} \choose j-i+1}$ ways, and every such path has weight $\frac{1}{\beta}$.

\item \emph{A path has one step of weight $\left(\frac{1}{\alpha}\right)^{\ell}$ and the rest of weight 1.} The path necessarily takes the first step east and goes to the right-most vertex of parallelogram number $\lambda_{i+1}+\ell-1$ in row $i-1$. This can happen in ${\lambda_{i+1}+\ell \choose j-i}$ ways, and every such path has weight $\left(\frac{1}{\alpha}\right)^{\ell}$, where $1\leq\ell\leq\lambda_i-\lambda_{i+1}$.

\item \emph{A path has one step of weight $\frac{1}{\beta}$, one step of weight $\left(\frac{1}{\alpha}\right)^{\ell}$, and the rest of weight 1.} The path necessarily takes the first step southeast and goes to the right-most vertex of parallelogram number $\lambda_{i+1}+\ell-1$ in row $i-1$. This can happen in ${\lambda_{i+1}+\ell \choose j-i-1}$ ways, and every such path has weight $\frac{1}{\beta}\left(\frac{1}{\alpha}\right)^{\ell}$, where $1\leq\ell\leq\lambda_i-\lambda_{i+1}$.
\end{enumerate}

We combine the above to obtain $A_{\lambda} = \left(w_{ij}\right)_{1 \leq i,j \leq k}$ as desired.

Finally, if $C$ is the Catalan path corresponding to the Catalan tableau $T$, since $\wt(T)= \wt(C) = (\alpha\beta)^n \pwt(C) = \beta^n\alpha^{k+\lambda_1} \wt(\textbf{P}(C))$, we obtain the desired formula.\qed

\noindent \textbf{Acknowledgement.} The author gratefully acknowledges Xavier G.~Viennot for enlightening conversations that inspired this work.

%thus A_lambda=desired.

%\section{Catalan tableaux with $k$ rows to $k$ non-intersecting paths}
%\begin{proof}
%From past results, 
%$$Prob(\tau) = \sum_T wt(T)$$
%where the sum is over all Catalan tableaux $T$ of type $\tau$.

%Using the bijections described, we convert from Catalan tableaux with $k$ rows to weighted Catalan paths on a Young diagram of $k$ rows to a set of $k$ non-intersecting weighted paths to obtain the following:
% $$Prob(\tau) = \sum_{p_{11},\ldots,p_{kk}} \prod_{i=1}^k wt(p_{ii})$$
%where the sum is over the $k$ non-intersecting paths $p_{11},\ldots,p_{kk}$ where $p_{ij}$ is a path taking the steps east and southeast from $e_i$ to $v_j$ as defined in Section \ref{paths_bij}. 

%By the LGV Theorem,
%$$\sum_{p_{11},\ldots,p_{kk}} \prod_{i=1}^k wt(p_{ii}) = \det A$$
%where $A = \begin{pmatrix} p_{11}&p{12}&\cdots&p_{1k}\\
%p_{21}&p_{22}&\cdots&p_{2k}\\
%\vdots&\vdots&\vdots&\vdots\\
%p_{k1}&p_{k2}&\cdots&p_{kk}.
%\end{pmatrix}
%$.

%From the above, $p_{ij}=0$ for $j<i-1$, $p_{ij}=BAH$ for $j=i-1$, and $p_{ij}=BLAH$ for $j \geq i$. Theorem \ref{deter} follows.
%\end{proof}


\begin{thebibliography}{99}

%\bibitem{williams07} S. Corteel, L. Williams. A Markov chain on permutations which projects to the PASEP.  International Mathematics Research Notices, article ID mm055 (2007).

\bibitem{williams2007} S.~Corteel and L.~Williams. Tableaux combinatorics for the asymmetric exclusion process, Advances in Applied Mathematics, Volume 39, Issue 3, 293--310 (2007).

\bibitem{williams2011} S.~Corteel and L.~Williams. Tableaux combinatorics for the asymmetric exclusion process and Askey-Wilson polynomials. Duke Math. J., 159: 385--415, (2011).

\bibitem{derrida} B.~Derrida, M.~Evans, V.~Hakim, V.~Pasquier. Exact solution of a 1D asymmetrix exclusion model using a matrix formulation, J. Phys. A: Math. Gen. 26, 1493--1517 (1993).

\bibitem{mandelshtam} O.~Mandelshtam. A Determinantal Formula for Catalan Tableaux and TASEP Probabilities, J. Combin. Theory Ser. A (2015).

\bibitem{duchi} E.\ Duchi and G.\ Schaeffer. A combinatorial approach to jumping particles, J. Combin. Theory Ser. A 110, 1--29 (2005).

%\bibitem{shapiro} L.\ Shapiro and D.\ Zeilberger. A Markov chain occurring in enzyme kinetics, J. Math. Biol. 15 351--357 (1982).

\bibitem{viennot} X.\,G.~Viennot. Canopy of binary trees, Catalan tableaux and the asymmetric exclusion process, FPSAC 2007, Formal Power Series and Algebraic Combinatorics (2007).

\bibitem{slides} X.\,G.~Viennot. Alternative tableaux, permutations and partially asymmetric exclusion process. Workshop ``Statistical Mechanics and Quantum-Field Theory Methods in Combinatorial Enumeration'', Isaac Newton Institute for Mathematical Science, Cambridge, 23 April 2008, video and slides available at: \url{http://sms.cam.ac.uk/media/1004}. 

\bibitem{slides2}  X.\,G.~Viennot. Forme des permutations, chemins et profil des arbres binaires, 52\`{e}me SLC, LascouxFest, Domaine Saint Jacques, Otrott, Mars 2004, slides available at: \url{http://www.xavierviennot.org/xavier/exposes_files/LascouxFest.pdf}.

%\bibitem{viennot2} Viennot, Catalan tableaux and the asymmetric exclusion process

%\bibitem{advances2005} L.\ Williams, Enumeration of totally positive Grassmann cells, Advances in Mathematics, Volume 190, Issue 2, 319--342 (2005).

\bibitem{nadeau} P.~Nadeau. The structure of alternative tableaux. J. Combin. Theory
Ser. A 118, no. 5, 1638--1660 (2011).

\bibitem{treelike} J.~Aval, A.~Boussicault, P.~Nadeau. Tree-like tableaux. Electron. J. Combin. 20, no. 4, Paper 34, 24 pp. (2013).

\bibitem{lgv} I.~Gessel, X.\,G.~Viennot. Determinants, Paths and Plane Partitions, Brandeis University, 36p. (1989).

%\bibitem{williams2010} S. Corteel, R. Stanley, D. Stanton, and L. Williams. Formulae for Askey-Wilson moments and enumeration of staircase tableaux. Transactions of the AMS 364 (2012), no. 11, 6009-6037.


%\bibitem{Hitcz} P. Hitczenko and S. Janson. Weighted random staircase tableaux. http://front.math.ucdavis.edu/1212.5498 (2012).

\end{thebibliography}
\end{document}